\documentclass[preprint,12pt]{elsarticle}

\makeatletter
\def\ps@pprintTitle{%
}
\makeatother

\usepackage{amsmath}
\usepackage{amssymb}
\usepackage{amsthm}
\usepackage{graphicx}
\usepackage{hyperref}
\usepackage{tikz}
\usetikzlibrary{fit, arrows, shapes, positioning, shadows, matrix, calc, positioning, shadows, angles, intersections}

\tikzset{cross/.style={cross out, draw=black, minimum size=2*(#1-\pgflinewidth), inner sep=0pt, outer sep=0pt},
cross/.default={4pt}}

\usetikzlibrary{decorations.pathreplacing}
\tikzstyle std=[line width=0.7pt]   
\tikzstyle stdthin=[line width=0.3pt]   
\tikzstyle stdthick=[line width=1.0pt]   

\tikzstyle fwd=[line width=0.7pt, ->]   
\tikzstyle fwdthin=[line width=0.3pt, ->]   
\tikzstyle fwdthick=[line width=1.0pt, ->]   
\tikzstyle fwddash=[line width=0.7pt, dashed, ->]   

\tikzstyle bwd=[double, line width=0.3pt, ->]  
\tikzstyle refl=[double,  dashed, line width=0.2pt, ->]    

\tikzstyle{every node}=[font=\small] 

\tikzset{
	>=stealth', 
	invisible/.style={opacity=0}, 
	alt/.code args={<#1>#2#3}{\alt<#1>{\pgfkeysalso{#2}}{\pgfkeysalso{#3}}}, 
	visible on/.style={alt=#1{}{invisible}}, 
	smallnode/.style={circle, fill=black, thick, inner sep=1pt, minimum size=1.5pt}, 
	punkt/.style={
			rectangle,
			rounded corners,
			draw=black, very thick,
			text width=5.5em,
			minimum height=2em,
			text centered},
	punkt_big/.style={
			rectangle,
			rounded corners,
			draw=black, very thick,
			text width=7em,
			minimum height=2em,
			text centered},
}

\usepackage{amsmath, amsfonts, mathtools, amssymb}

\usepackage[color,final]{showkeys}

\newcommand{\cut}[1]{{}}

\makeatletter
\newcommand*{\rom}[1]{\expandafter\@slowromancap\romannumeral #1@}
\makeatother

\newcommand{\cA}{{\mathcal{A}}}
\newcommand{\cB}{{\mathcal{B}}}
\newcommand{\cC}{{\mathcal{C}}}
\newcommand{\cD}{{\mathcal{D}}}

\newcommand{\cF}{{\mathcal{F}}}
\newcommand{\cG}{{\mathcal{G}}}
\newcommand{\cH}{{\mathcal{H}}}

\newcommand{\cZ}{{\mathcal{Z}}}

 \usepackage[bb=boondox]{mathalfa}

\usepackage{xparse}
\DeclareFontFamily{U}{ntxmia}{}
\DeclareFontShape{U}{ntxmia}{m}{it}{<-> ntxmia }{}
\DeclareFontShape{U}{ntxmia}{b}{it}{<-> ntxbmia }{}
\DeclareSymbolFont{lettersA}{U}{ntxmia}{m}{it}
\SetSymbolFont{lettersA}{bold}{U}{ntxmia}{b}{it}

\ExplSyntaxOn
\NewDocumentCommand{\varmathbb}{m}
 {
  \tl_map_inline:nn { #1 }
  {
    \use:c { varbb##1 }
  }
 }
\tl_map_inline:nn { ABCDEFGHIJKLMNOPQRSTUVWXYZ }
 {
  \exp_args:Nc \DeclareMathSymbol{varbb#1}{\mathord}{lettersA}{\int_eval:n { `#1+67 }}
 }
\exp_args:Nc \DeclareMathSymbol{varbbk}{\mathord}{lettersA}{169}
\ExplSyntaxOff

\DeclareMathOperator*{\argmin}{argmin}

\DeclarePairedDelimiter{\dotp}{\langle}{\rangle}

\renewcommand{\Re}{\operatorname{Re}} 	

\newcommand{\prox}{\mathrm{Prox}}

\newcommand{\dom}{\mathrm{dom}\,} 

\newcommand{\reals}{\mathbb{R}}
\newcommand{\complex}{\mathbb{C}}
\newcommand{\binfty}{{\boldsymbol \infty}}

\makeatletter
\DeclareFontFamily{U}{tipa}{}
\DeclareFontShape{U}{tipa}{m}{n}{<->tipa10}{}
\newcommand{\arc@char}{{\usefont{U}{tipa}{m}{n}\symbol{62}}}%

\newcommand{\arc}[1]{\mathpalette\arc@arc{#1}}

\newcommand{\arc@arc}[2]{%
  \sbox0{$\m@th#1#2$}%
  \vbox{
    \hbox{\resizebox{\wd0}{\height}{\arc@char}}
    \nointerlineskip
    \box0
  }%
}
\makeatother

\definecolor{lightgrey}{gray}{0.8}
\definecolor{medgrey}{gray}{0.6}
\definecolor{darkgrey}{gray}{0.4}

\newcommand{\opA}{{\varmathbb{A}}}
\newcommand{\opB}{{\varmathbb{B}}}
\newcommand{\opC}{{\varmathbb{C}}}

\newcommand{\opI}{{\varmathbb{I}}}
\newcommand{\opJ}{{\varmathbb{J}}}

\newcommand{\opT}{{\varmathbb{T}}}

\newcommand{\Di}{{\mathrm{Disk}}}
\newcommand{\Ci}{{\mathrm{Circ}}}

\DeclarePairedDelimiter{\norm}{\lVert}{\rVert}
\DeclarePairedDelimiter{\abs}{\lvert}{\rvert}
\DeclarePairedDelimiter{\babs}{\bigg\lvert}{\bigg\rvert}

\makeatletter
\newlength{\doublefracgap}
\DeclareRobustCommand{\doublefrac}[2]{%
  \mathinner{\mathpalette\doublefrac@{{#1}{#2}}}%
}
\newcommand{\doublefrac@}[2]{\doublefrac@@#1#2}
\newcommand{\doublefrac@@}[3]{%
  \ooalign{%
    \raisebox{\doublefracgap}{$\m@th#1\frac{#2}{\phantom{#3}}$}\cr
    \raisebox{-\doublefracgap}{$\m@th#1\frac{\phantom{#2}}{#3}$}\cr
  }%
}

\makeatother

\newcommand{\para}[1]{\left({#1}\right)}

\newcommand{\R}{\mathbb{R}}
\newcommand{\resa}[1]{\opJ_{\alpha #1}}
\newcommand{\srg}[1]{\cG\left(#1\right)}

\newcommand{\diffang}[4]{%
  	\begingroup 
		\def\firstop{#1}%
		\def\secondop{#2}%
		\def\identity{\opI}%
		\ifx\firstop\identity
		\def\firstop{}%
		\fi
		\ifx\secondop\identity
		\def\secondop{}%
		\fi
		\angle (\firstop #3 - \firstop #4, \secondop #3 - \secondop #4)
  	\endgroup
}
\newcommand{\diffnormfrac}[4]{%
  	\begingroup 
		\def\firstop{#1}%
		\def\secondop{#2}%
		\def\identity{\opI}%
		\ifx\firstop\identity
		\def\firstop{}%
		\fi
		\ifx\secondop\identity
		\def\secondop{}%
		\fi
		\frac{\norm{\firstop #3 - \firstop #4}}{\norm{\secondop #3 - \secondop #4}}
  	\endgroup
}

\newcommand{\proxa}[1]{\prox_{\alpha #1}}

\newtheorem{theorem}{Theorem}
\newtheorem{fact}{Fact}

\journal{Journal of Mathematical Analysis and Applications}

\begin{document}

\begin{frontmatter}



\title{Convergence Analyses of Davis--Yin Splitting via Scaled Relative Graphs II:\\
Convex Optimization Problems}


\author[cmu]{Soheun Yi}
\ead{soheuny@andrew.cmu.edu}

\author[ucla]{Ernest K. Ryu\corref{cor1}}
\ead{eryu@math.ucla.edu}

\cortext[cor1]{Corresponding author}

\affiliation[cmu]{
	organization={Carnegie Mellon University, Department of Statistics and Data Science},
	addressline={5000 Forbes Avenue},
	city={Pittsburgh},
	postcode={15213},
	state={PA},
	country={USA}
}
\affiliation[ucla]{
organization={University of California, Los Angeles, Department of Mathematics},
addressline={520 Portola Plaza}, 
city={Los Angeles},
postcode={90095}, 
state={CA},
country={USA}}

\begin{abstract}
    The prior work of [SIAM J. Optim., 2025] used scaled relative graphs (SRG) to analyze the convergence of Davis--Yin splitting (DYS) iterations on monotone inclusion problems. In this work, we use this machinery to analyze DYS iterations on convex optimization problems and obtain state-of-the-art linear convergence rates.
\end{abstract}








\end{frontmatter}


\section{Introduction}
Consider the problem 
	\begin{equation}\label{eq:primal}
 \begin{array}{ll}
\underset{x\in\cH}{\mbox{minimize}}&f(x) + g(x) + h(x),
 \end{array}
	\end{equation}
	where $\cH$ is a Hilbert space, $f$, $g$, and $h$ are convex, closed, and proper functions, and $h$ is differentiable with $L$-Lipschitz continuous gradients.
 The Davis--Yin splitting (DYS) \cite{davis_three-operator_2017} solves this problem by performing the fixed-point iteration with 
	\begin{equation}\label{eq:dys-subdiff-1}
		\opT = \opI - \proxa{g} + \proxa{f} (2\proxa{g} - \opI - \alpha \nabla h \proxa{g}),
	\end{equation}
	where $\alpha > 0$,
  $\prox_{\alpha f}$ and $\prox_{\alpha g}$ are the proximal operators with respect to $\alpha f$ and $\alpha g$, and $\opI$ is the identity mapping.
  DYS has been used as a building block for various algorithms for a diverse range of optimization problems \cite{yan2018new,wang2017three,carrillo2022primal,van2020three, weylandt2019splitting,heaton2021learn}.

 Much prior work has been dedicated to analyzing the convergence rate of DYS iterations  \cite{davis_three-operator_2017, Aragon-ArtachoTorregrosa-Belen2022_direct, DaoPhan2021_adaptive, pedregosa_convergence_2016, ryu_operator_2020, wang_robust_2019, CondatRichtarik2022_randprox}.	Recently, Lee, Yi, and Ryu \cite{LeeYiRyu2022_convergencea} leveraged the recently introduced scaled relative graphs (SRG) \cite{ryu_scaled_2021} to obtain tighter analyses.
	However, the focus of \cite{LeeYiRyu2022_convergencea} was on DYS applied to the general class of monotone operators, rather than the narrower class of subdifferential operators.

In this paper, we use the SRG theory of \cite{LeeYiRyu2022_convergencea} to analyze the linear convergence rates of DYS applied to convex optimization problems and obtain state-of-the-art rates.


	\subsection{Prior works}
The theory of monotone operators and splitting methods is a powerful tool for deriving and analyzing a wide range of convex optimization algorithms \cite{bauschke_2017_convex,ryu_primer_2016, ryu2021large}. Widely used splitting methods include forward-backward splitting (FBS) \cite{bruck1977weak, passty1979ergodic}, Douglas--Rachford splitting (DRS) \cite{peaceman1955numerical, douglas1956numerical, lions1979splitting}, and alternating directions method of multipliers (ADMM) \cite{gabay1976dual}.
The Davis--Yin splitting (DYS) \cite{davis_three-operator_2017} applies to finding a zero for the sum of three monotone operators and unifies the prior two-operator splitting methods FBS and DRS. The DYS splitting method has a variety of applications \cite{yan2018new, wang2017three, carrillo2022primal, van2020three, weylandt2019splitting, heaton2021learn} and many variants, including stochastic DYS \cite{yurtsever_three_2021, yurtsever2016stochastic, cevher2018stochastic, yurtsever_three_2021-1, PedregosaFatrasCasotto2019_proximal}, inexact DYS \cite{zong2018convergence}, adaptive DYS  \cite{pedregosa_adaptive_2018}, inertial DYS \cite{Cui2019AnIT}, and primal-dual DYS \cite{salim2022dualize} have been proposed and studied.

However, while there has been a relatively large body of research on the various applications and variants of DYS and their (sublinear) convergence,  there is not much literature on linear convergence analysis of the DYS iteration. Among the few prior work, one approach formulates SDPs that numerically computed the tight contraction factors of DYS: Ryu, Taylor, Bergeling, and Giselsson \cite{ryu_operator_2020} and Wang, Fazlyab, Chen, and Preciado \cite{wang_robust_2019} carried out this approach using the performance estimation problem (PEP) and integral quadratic constraint (IQC), respectively. However, this approach does not lead to an analytical expression of the contraction factors. There are only a handful of prior works providing analytical expression of the contraction factors for DYS. The work by Davis and Yin \cite{davis_three-operator_2017}, and Condat and Richt{\'a}rik \cite{CondatRichtarik2022_randprox} obtain analytical contraction factors through standard analyses.  One more work by Lee, Yi, and Ryu \cite{LeeYiRyu2022_convergencea} takes a different approach and uses the machinery of scaled relative graphs.
	
	This novel tool, the scaled relative graphs (SRG) \cite{ryu_scaled_2021}, provides a new approach to analyzing the behavior of multi-valued (non-linear) operators by mapping their action onto the (extended) complex plane. This theory was further studied and utilized by Huang, Ryu, and Yin \cite{huang_scaled_2019}, who identified the SRG of normal matrices; Pates, who leveraged the Toeplitz--Hausdorff theorem to identify SRGs of linear operators \cite{pates_scaled_2021}; and Huang, Ryu, and Yin, who used the SRG to prove the tightness of Ogura and Yamada's  \cite{ogura_non-strictly_2002} averagedness coefficients of the composition of averaged operators.	Moreover, the SRG has been utilized in control theory by Chaffey, Forni, and Rodolphe to examine input-output properties of feedback systems \cite{ChaffeyForniSepulchre2023_graphical,chaffey_scaled_2021}, and Chaffey and Sepulchre have further found its application to characterize behaviors of a given model by leveraging it as an experimental tool \cite{ChaffeySepulchre2024_monotone,chaffey_rolled-off_2021,ChaffeyPadoan2022_circuit}.

	\subsection{Preliminaries}

	\paragraph{Multi-valued operators} 
    In general, we follow notations regarding multi-valued operators presented in \cite{bauschke_2017_convex,ryu2021large}.
	Write $\cH$ for a real Hilbert space with inner product $\dotp{\cdot, \cdot}$ and norm $\norm{\cdot}$. 
    To represent that $\opA$ is a multi-valued operator defined on $\cH$, write $\opA \colon \cH \rightrightarrows \cH$, and define its domain as $\dom\opA = \left\{ x \in \cH \,|\, \opA x \neq \emptyset \right\}$.
	We say $\opA$ is single-valued if all outputs of $\opA$ are singletons or the empty set, and identify $\opA$ with the function from $\dom\opA$ to $\cH$.
	Define the graph of an operator $\opA$ as
	\[
		\mathrm{graph}(\opA) = \{(x, u) \in \cH \times \cH \,|\, u \in \opA x\}.
	\]
	We do not distinguish $\opA$ and $\mathrm{graph}(\opA)$ for the sake of notational simplicity. 
	For instance, it is valid to write $(x,u) \in \opA$ to mean $u \in \opA x$.
	Define the inverse of $\opA$ as
	\[
		\opA ^{-1} = \{(u, x) \,|\, (x,u) \in \opA\},
	\]
	scalar multiplication with an operator as
	\[
		\alpha\opA  = \{(x,\alpha u) \,|\, (x, u) \in \opA\},
	\]
	the identity operator as
	\[
	    \opI = \{ (x, x) \,|\, x \in \cH \},
	\]
	and 
	\[
		\opI + \alpha\opA = \{ (x, x + \alpha u) \,|\, (x, u) \in \opA \}
	\]
	for any $\alpha\in\reals$.
	Define the resolvent of $\opA$ with stepsize $\alpha>0$ as 
	\[
	    \opJ_{\alpha \opA} = (\opI +\alpha \opA)^{-1}.
	\]
	Note that $\opJ_{\alpha \opA}$ is a single-valued operator if $\opA$ is monotone, or equivalently if $\dotp{x - y, u - v} \ge 0$ for all $(x,u)$, $(y,v) \in \opA$.
	Define addition and composition of operators $\opA \colon \cH \rightrightarrows \cH$ and $\opB \colon \cH \rightrightarrows \cH$ as
	\begin{align*}
		\opA + \opB &= \left\{ (x, u + v) \,|\, (x, u) \in \opA, (x, v) \in \opB \right\},\\
		\opA \opB &= \left\{ (x, s) \,|\, \exists \, u \; \text{such that} \; (x, u) \in \opB, (u, s) \in \opA \right\}.
	\end{align*}

	We call $\cA$ a class of operators if it is a set of operators.
	For any real scalar $\alpha\in \reals$, define
        \[
            \alpha \cA = \{ \alpha \opA \,|\, \opA \in \cA\}
        \]
	and \[
	   \opI + \alpha\cA = \{ \opI + \alpha\opA \,|\, \opA \in \cA\}.
	\]
	 Define 
	\[
		\cA^{-1} = \{\opA^{-1} \,|\, \opA \in \cA \}
	\]
	and $\opJ_{\alpha\cA} = (\opI + \alpha\cA)^{-1}$ for $\alpha > 0$.

    \paragraph{Subdifferential operators}
    Unless otherwise stated, functions defined on $\cH$ are extended real-valued, which means
	\[
        f \colon \cH \to \R \cup \left\{ \pm \infty \right\}.
    \]
	For a function $f$, we define the subdifferential operator $\partial f$ via
	\[
	    \partial f(x) = \left\{ g \in \cH \,|\, f(y) \ge f(x) + \dotp{g, y - x}, \forall y \in \cH \right\}
	\]
	(we allow $\infty \ge \infty$ and $-\infty \ge -\infty$). 
	In some cases, the subdifferential operator $\partial f$ is a single-valued operator. 
	Then, we write $\nabla f = \partial f$.

    \paragraph{Proximal operators}
    We call $f$ a CCP function if it is convex, closed, and proper \cite{ryu2021large,bauschke_2017_convex}.
    For a CCP function $f \colon \cH \to \R \cup \left\{ \pm \infty \right\}$ and $\alpha > 0$, we define the proximal operator with respect to $\alpha f$ as 
    \[
        \prox_{\alpha f}(x) = \argmin_{y \in \cH} \left\{ \alpha f(y) + \frac{1}{2} \norm{x - y}^2 \right\}.
    \]
    Then, $\resa{\partial f} = \prox_{\alpha f}$.
    
	\paragraph{Class of functions and subdifferential operators}
        Define $f \colon \cH \to \R \cup \left\{ \pm \infty \right\}$ being $\mu$-strongly convex (for $\mu \in (0, \infty)$) and $L$-smooth (for $L \in (0, \infty)$) as they are defined in \cite{Nesterov2014_introductory}.
	Write
	\[
		\cF_{\mu, L} = \left\{ f \,|\, f \text{ is } \mu \text{-strongly convex, } L \text{-smooth, and CCP.} \right\}.
	\]
	for collection of functions that are $\mu$-strongly convex and $L$-smooth at the same time.
	For notational simplicity, we extend $\cF_{\mu, L}$ to allow $\mu = 0$ or $L = \infty$ by defining
	\begin{align*}
		\cF_{0, L} &= \left\{ f \,|\, f \text{ is } L \text{-smooth and CCP.} \right\},\\
		\cF_{\mu, \infty} &= \left\{ f \,|\, f \text{ is } \mu \text{-strongly convex and CCP.} \right\},\\
		\cF_{0, \infty} &= \left\{ f \,|\, f \text{ is CCP.} \right\}.
	\end{align*}
	for $\mu$, $L \in (0, \infty)$.

	Subdifferential operators of any functions in $\cF_{\mu, L}$ are denoted 
	\[
		\partial \cF_{\mu, L} = \left\{ 
		\partial f \,|\, f \in \cF_{\mu, L} \right\}.
	\]
	
	\paragraph{Complex set notations}
	Denote $\overline{\complex} = \complex \cup \left\{ \infty \right\}$, and define $0^{-1} = \infty$ and $\infty^{-1} = 0$ in $\overline{\complex}$. 
	For $A \subset \overline{\complex}$ and $\alpha \in \complex$, define
    \[
        \alpha A = \left\{ \alpha z \,|\, z \in A \right\}, 
        \quad \alpha + A = \left\{ \alpha + z \,|\, z \in A \right\},
        \quad A^{-1} = \left\{ z^{-1} \,|\, z \in A \right\}.
    \]
	For $A \subseteq \complex$, define the boundary of $A$ 
	\[
        \partial A = \overline{A}\setminus \mathrm{int} A.
    \]
	We clarify that the usage of $\partial$ operator is different  when it is applied to a function or a complex set; the former is the subdifferential operator, and the latter is the boundary operator.
	For circles and disks on the complex plane, write
	\[
	    \Ci(z, r) = \{ w \in \complex \,|\, \abs{w - z} = r\},
	    \qquad
	    \Di(z, r) = \{ w \in \complex \,|\, \abs{w - z} \le r \}
	\]
	for $z \in \complex$ and $r \in (0, \infty)$. 
	Note relationship that $\Ci(z, r) = \partial \Di(z, r)$.
	In this paper, the $z$ in $\Ci(z, r) $ are real numbers without a complex part.

	\paragraph{Scaled relative graphs \cite{ryu_scaled_2021}} 

	Define the SRG of an operator $\opA\colon\cH\rightrightarrows\cH$ as
	\begin{align*}
		\mathcal{G}(\opA)&=
		\left\{
		\frac{\|u-v\|}{\|x-y\|}
		\exp\left[\pm i \angle (u-v,x-y)\right]
		\,\Big|\,
		u\in \opA x,\,v\in \opA y,\, x\ne y \right\}\\
		&\qquad\qquad\qquad\qquad\qquad\qquad\qquad
		\bigg(\cup \{\infty\}\text{ if $\opA$ is not single-valued}\bigg).
	\end{align*}
	where the angle between $x \in \cH$ and $y \in \cH$ is defined as
	\begin{align*}
		\angle(x,y) =
		\left\{
		\begin{array}{ll}
		\arccos\para{\tfrac{\dotp{x,y}}{\norm{x}\norm{y}}}&\text{ if }x\ne 0,\,y\ne 0\\
		0&\text{ otherwise.}
		\end{array}
		\right.
	\end{align*}
        Note, SRG is a subset of $\overline{\complex}$.
	Define the SRG of a class of operators $\cA$ as
	\[
	\mathcal{G}(\cA)=\bigcup_{\opA\in \cA}\mathcal{G}(\opA).
	\]
	We say $\cA$ is \emph{SRG-full} if
	\begin{align*}
	\opA\in \cA
	\quad\Leftrightarrow\quad
	\cG(\opA)\subseteq\cG(\cA),
	\end{align*}
	which essentially means that the membership in an SRG-full class is entirely characterized by its SRG, providing one-to-one correspondence between geometric operations in the language of SRGs and operator algebra.
	The following fact states that SRG-fullness is invariant under common operations on operators.

    \begin{fact}[Theorem 4, 5 \cite{ryu_scaled_2021}]
        \label{fact:srg-trans}
             If $\cA$ is a class of operators, then 
        \[
        \cG(\alpha \cA)=\alpha\cG(\cA), \quad
        \cG(\opI+ \cA)=1+\cG(\cA), \quad
        \cG(\cA^{-1}) = \cG(\cA)^{-1}.
        \]
        where $\alpha$ is a nonzero real number.
        If $\cA$ is furthermore SRG-full, then $\alpha \cA, \opI+ \cA$, and $\cA^{-1}$
        are SRG-full.
    \end{fact}

	\begin{fact}[Proposition 2 \cite{ryu_scaled_2021}]
		\label{fact:cvx-srg-first}
		Let $0<\mu<L<\infty$. Then
		\vspace{0.05in}
		\begin{center}
		\begin{tabular}{cc}
        \multicolumn{1}{l}{$\cG(\partial \mathcal{F}_{0,\infty}) =\left\{z\,|\,\Re z\ge 0\right\}\cup\{\infty\}$}
        &
        \multicolumn{1}{l}{
        $\cG(\partial \mathcal{F}_{\mu,\infty})=\left\{z\,|\,\Re z\ge \mu\right\}\cup\{\infty\}$}
        \\
		\raisebox{-0.5\height}{
		\begin{tikzpicture}[scale=1, baseline=(current bounding box.center)]
		\fill[fill=medgrey] (0,-1) rectangle (1.4,1);
		\draw [<->] (-0.5,0) -- (1.4,0);
		\draw [<->] (0,-1) -- (0,1);
		\draw (0.95,.8) node {$\scriptstyle\cup \{\binfty\}$};
		\end{tikzpicture}}
		&
		\raisebox{-.5\height}{
		\begin{tikzpicture}[scale=1, baseline=(current bounding box.center)]
		\draw [<->] (0,-1) -- (0,1);
		\fill[fill=medgrey] (0.2,-1) rectangle (1.4,1);
		\draw [<->] (-0.5,0) -- (1.4,0);
		\draw (0.95,.8) node {$\scriptstyle \cup \{\binfty\}$};
		\draw (0.2,0pt) node [above right] {$\mu$};
		\filldraw (.2,0) circle ({0.6*1.5/1pt});
		\end{tikzpicture}}
        \\[2cm]
        \multicolumn{1}{l}{
        $\cG(\partial \mathcal{F}_{0,L})=\Di(L/2, L/2)$}
        &
        \multicolumn{1}{l}{
        $\cG(\partial \mathcal{F}_{\mu,L})=\Di((L + \mu)/2, (L - \mu)/2)$}
        \\
		\raisebox{-.5\height}{
		\begin{tikzpicture}[scale=1, baseline=(current bounding box.center)]
		\fill[fill=medgrey] (0.5,0) circle (0.5);
		\draw [<->] (0,-1.0) -- (0,1.0);
		\draw (0.9,0) node [above right] {$L$};
		\filldraw (1,0) circle ({0.6*1.5/1pt});
		\draw [<->] (-0.5,0) -- (1.4,0);
		\end{tikzpicture}}
		&
		\raisebox{-.5\height}{
		\begin{tikzpicture}[scale=1, baseline=(current bounding box.center)]
		\draw [<->] (0,-1.0) -- (0,1.0);
		\fill[fill=medgrey] (0.65,0) circle (0.35);
		\draw [<->] (-0.5,0) -- (1.4,0);
		\filldraw (1,0) circle ({0.6*1.5/1pt});
		\draw (0.9,0pt) node [above right] {$L$};
		\filldraw (.3,0) circle ({0.6*1.5/1pt});
		\draw (0.45,0) node [above left] {$\mu$};
		\end{tikzpicture}
		}
		\end{tabular}
		\end{center}
		\vspace{0.05in}
	\end{fact}

	\paragraph{DYS operators}
	Let
	\begin{align*}
		\opT_{\opA, \opB, \opC, \alpha, \lambda} 
		&=
		\opI - \lambda \resa{\opB} + \lambda \resa{\opA} (2\resa{\opB} - \opI - \alpha \opC \resa{\opB}) 
	\end{align*}
	be the DYS operator for operators $\opA\colon \cH\rightrightarrows\cH$, $\opB\colon \cH\rightrightarrows\cH$, and $\opC\colon \cH\rightrightarrows\cH$ 
	with stepsize $\alpha \in (0, \infty)$ and averaging parameter $\lambda \in (0, \infty)$.
	In this paper, we usually take $\opA = \partial f$, $\opB = \partial g$, and $\opC = \nabla h$ for some CCP functions $f$, $g$, and $h$ defined on $\cH$, to obtain
	\[
		\opT_{\partial f, \partial g, \nabla h, \alpha, \lambda}
		= \opI - \lambda \proxa{g} + \lambda \proxa{f} (2\proxa{g} - \opI - \alpha \nabla h \proxa{g})
	\]
	what we call the \emph{subdifferential DYS operator}.
	
	Let 
	\[
		\opT_{\cA, \cB, \cC, \alpha, \lambda} = \left\{ 
		\opT_{\opA, \opB, \opC, \alpha, \lambda} \, |\, \opA \in \cA, \opB \in \cB, \opC \in \cC
			\right\}
	\]
	 be the class of DYS operators for operator classes $\cA$, $\cB$, and $\cC$ with $\alpha, \lambda \in (0, \infty)$.
	Define
	\begin{align*}
		\zeta_{\text{DYS}}(z_A, z_B, z_C; \alpha, \lambda) 
		&= 1 - \lambda z_B + \lambda z_A (2z_B - 1 - \alpha z_C z_B) \\
		&= 1 - \lambda z_A - \lambda z_B + \lambda (2 - \alpha z_C) z_A z_B,
	\end{align*}
	which exhibits symmetry with respect to $z_A$ and $z_B$, and
	\[
		\cZ^{\text{DYS}}_{\cA, \cB, \cC, \alpha, \lambda} = 
		\left\{ \zeta_{\text{DYS}}(z_A, z_B, z_C; \alpha, \lambda)
		\,|\, z_A \in \srg{\resa{\cA}}, z_B \in \srg{\resa{\cB}}, z_C \in \srg{\cC}
		\right\}
	\]
	for operator classes $\cA$, $\cB$, and $\cC$ with $\alpha, \lambda \in (0, \infty)$.

	\paragraph{Identifying the tight Lipschitz coefficient via SRG}
    We say a subset of $\overline{\complex}$ is a \emph{generalized disk} if it is a disk, $\{ z \,|\, \Re z \ge a \} \cup \{\infty\}$, or $\{ z \,|\, \Re z \le a \} \cup \{\infty\}$ for a real number $a$.
    The following is the key fact for calculating the Lipschitz coefficients of the DYS operators via SRG. 
    
	\begin{fact}[Corollary~2.2 of \cite{LeeYiRyu2022_convergencea}]\label{fact:tight-coeff-dys}
		Let $\alpha, \lambda > 0$. 
		Let $\cA$ and $\cB$ be SRG-full classes of monotone operators where $\srg{\opI + \alpha\cA}$ forms a generalized disk.
		Let $\cC$ be an SRG-full class of single-valued operators with $\srg{\cC}$ being a generalized disk.
		Assume $\srg{\cA}$, $\srg{\cB}$, and $\srg{\cC}$ are nonempty.
		Then,
		\[
			\sup_{\substack{\opT \in \opT_{\cA, \cB, \cC, \alpha, \lambda} \\ x, y \in \dom\opT, x \ne y}}
			\frac{\norm{\opT x - \opT y}}{\norm{x - y}}
			= \sup_{z \in \cZ^{\text{DYS}}_{\cA, \cB, \cC, \alpha, \lambda}} \abs{z}.
		\]
	\end{fact}
    In fact, the original version of Fact~\ref{fact:tight-coeff-dys} allows $\srg{\opI + \alpha\cA}$ to have a more general property, namely the so-called ``arc property.''
    We can calculate bounds for the right-hand-side of the equality in Fact~\ref{fact:tight-coeff-dys} efficiently by using the following fact.
	\begin{fact}[Lemma~1.1 of \cite{LeeYiRyu2022_convergencea}]\label{fact:max-mod}
		Let $f \colon \complex^3 \to \complex$ be a polynomial of three complex variables.
		Let $A$, $B$, and $C$ be compact subsets of $\,\complex$.
		Then, 
		\[
			\max_{\substack{z_A \in A , z_B \in B ,\\ z_C \in C}} \abs{f(z_A, z_B, z_C)}
			=
			\max_{\substack{z_A \in \partial A , z_B \in \partial B ,\\ z_C \in \partial C}} \abs{f(z_A, z_B, z_C)}.
		\]
	\end{fact}	
	
	\section{Contraction factors of DYS for convex optimization problems}
	We now present Lipschitz factors of DYS for convex optimization problems. When the Lipschitz factor is strictly less than $1$, we, of course, have a strict contraction.

To the best of our knowledge, the convergence rates provided by our Theorems~\ref{thm:main-result-1} and \ref{thm:main-result-2} are the best linear convergence rates in the sense that they are not slower than the prior rates in all cases and faster in most cases. We provide specific comparisons against prior rates in Section~\ref{ss:comparison}.


    \begin{theorem}\label{thm:main-result-1}
        Let $f \in \cF_{\mu_f, L_f}$, $g \in \cF_{\mu_g, L_g}$, and $h \in \cF_{\mu_h, L_h}$, where 
        \begin{gather*}
            0 \le \mu_f < L_f \le \infty, \quad 0 \le \mu_g < L_g \le \infty, \quad 0 \le \mu_h < L_h < \infty.
        \end{gather*}
        Let $\lambda > 0$ be an averaging parameter and $\alpha > 0$ be a step size.
        Throughout this theorem, define $r / \infty = 0$ for any real number $r$. Write
        \begin{gather*}
			d = \max\{\abs{2 - \lambda - \alpha\mu_h}, \abs{2 - \lambda - \alpha L_h}\}, \\
            C_f = \frac{1}{2}\left( \frac{1}{1 + \alpha \mu_f} + \frac{1}{1 + \alpha L_f} \right), \quad
            C_g = \frac{1}{2}\left( \frac{1}{1 + \alpha \mu_g} + \frac{1}{1 + \alpha L_g} \right), \\
			R_f = \frac{1}{2}\left( \frac{1}{1 + \alpha \mu_f} - \frac{1}{1 + \alpha L_f} \right), \quad
            R_g = \frac{1}{2}\left( \frac{1}{1 + \alpha \mu_g} - \frac{1}{1 + \alpha L_g} \right).
        \end{gather*}
        If $\lambda < 1 / C_f$, then $\opT_{\partial f, \partial g, \nabla h, \alpha, \lambda}$ is $\rho_f$-Lipschitz, where
        \begin{align*}
            &\rho_f^2 = 
				\left( 1 - \lambda \frac{C_f^2 - R_f^2}{C_f} \right) 
				\max\bigg\{ 
                \left( 1 - \frac{\lambda}{1 + \alpha\mu_g} \right)^2 + 
                \frac{\lambda d^2}{1 / C_f - \lambda} \left( \frac{1}{1 + \alpha\mu_g} \right)^2, \\
                &\qquad\qquad\qquad\qquad\qquad\qquad\quad
                \left( 1 - \frac{\lambda}{1 + \alpha L_g} \right)^2 + 
                \frac{\lambda d^2}{1 / C_f - \lambda} \left( \frac{1}{1 + \alpha L_g} \right)^2
            \bigg\}.
        \end{align*}
        Symmetrically, if $\lambda < 1 / C_g$, then $\opT_{\partial f, \partial g, \nabla h, \alpha, \lambda}$ is $\rho_g$-Lipschitz, where
        \begin{align*}
            &\rho_g^2 = 
				\left( 1 - \lambda \frac{C_g^2 - R_g^2}{C_g} \right) 
				\max\bigg\{ 
                \left( 1 - \frac{\lambda}{1 + \alpha\mu_f} \right)^2 + 
                \frac{\lambda d^2}{1 / C_g - \lambda} \left( \frac{1}{1 + \alpha\mu_f} \right)^2, \\
                &\qquad\qquad\qquad\qquad\qquad\qquad\quad
                \left( 1 - \frac{\lambda}{1 + \alpha L_f} \right)^2 + 
                \frac{\lambda d^2}{1 / C_g - \lambda} \left( \frac{1}{1 + \alpha L_f} \right)^2
            \bigg\}.
        \end{align*}
    \end{theorem}

	\begin{theorem}\label{thm:main-result-2}
		Let $f$, $g$, $h$, $\mu_f$, $L_f$, $\mu_g$, $L_g$, $\mu_h$, $L_h$, $\lambda$, and $\alpha$ be the same as in Theorem~\ref{thm:main-result-1}.
		Additionally, assume $\lambda < 2 - \frac{\alpha (\mu_h + L_h)}{2}$.
            Write
		\begin{gather*}
			\nu_f = \min\left\{ 
			\frac{2 \mu_f + \mu_h}{(1 + \alpha \mu_f)^2}, \frac{2 L_f + \mu_h}{(1 + \alpha L_f)^2} \right\},\quad
			\nu_g = \min\left\{ 
			\frac{2 \mu_g + \mu_h}{(1 + \alpha \mu_g)^2}, \frac{2 L_g + \mu_h}{(1 + \alpha L_g)^2} \right\},
			\\
			\theta = \frac{2}{4 - \alpha(\mu_h + L_h)},
		\end{gather*}
        where we define $\infty / \infty^2 = 0$ so that $\nu_f = 0$ when $L_f = \infty$ and $\nu_g = 0$ when $L_g = \infty$.
		Then, $\opT_{\partial f, \partial g, \nabla h, \alpha, \lambda}$ is $\rho$-contractive, where
		\begin{align*}
			\rho^2 = 1 - \lambda \theta + \lambda \sqrt{\left(\theta - \alpha \nu_f \right) \left(\theta - \alpha \nu_g \right)}.
		\end{align*}
	\end{theorem}

		We remark that linear convergence of the DYS is implied only when $\min\{L_f,L_g\}<\infty$ and $\max\{\mu_f,\mu_g\}>0$. 
		If these conditions are not satisfied, Theorems~\ref{thm:main-result-1} and \ref{thm:main-result-2} yield a contraction factor of \(1\), which does not imply linear convergence.

	\subsection{\texorpdfstring{Proofs of Theorems~\ref{thm:main-result-1} and \ref{thm:main-result-2}}{Proofs}}
	
	To apply Fact~\ref{fact:tight-coeff-dys}, we need $\cA$, $\cB$, and $\cC$ to be SRG-full operator classes. While the choices
	\[
		\cA = \partial \cF_{\mu_f, L_f}, 
		\quad
		\cB = \partial \cF_{\mu_g, L_g},
		\quad
		\cC = \partial \cF_{\mu_h, L_h}
	\]
    would be natural, these classes are not SRG-full.
	Therefore we introduce the following operator classes:
	\begin{align*}
		\cD_f &= \{ \opA \colon \cH \rightrightarrows \cH \,|\, 
			\cG(\opA) \subseteq \cG(\partial \cF_{\mu_f, L_f})
		\}, \\
		\cD_g &= \{ \opB \colon \cH \rightrightarrows \cH \,|\, 
			\cG(\opB) \subseteq \cG(\partial \cF_{\mu_g, L_g})
		\}, \\
		\cD_h &= \{ \opC \colon \cH \rightrightarrows \cH \,|\, 
			\cG(\opC) \subseteq \cG(\partial \cF_{\mu_h, L_h})
		\}.
	\end{align*}
	To elaborate, we gather all operators that have their SRG within $\cG(\partial \cF_{\mu_f, L_f})$ to form $\cD_f$, and so on. 
    Then, $\cD_f$, $\cD_g$, and $\cD_h$ are SRG-full classes by definition.
	We now consider $\cA = \cD_f$, $\cB = \cD_g$, and $\cC = \cD_h$ in the following proof.
	
    We quickly mention two elementary facts.
	\begin{fact}
		\label{fact:CS}
			For $a,b,c,d\in[0,\infty)$,
			\[
			\left(\sqrt{ab} + \sqrt{cd}\right)^2 \le (a + c)(b + d).
			\]
	\end{fact}

         \begin{proof}
             This inequality is an instance of Cauchy--Schwarz.
         \end{proof}
			
	\begin{fact}\label{fact:distance-to-circ}
	Let $k$, $l$, and $r$ be positive real numbers, and $b$, $c$ be real numbers.
	For $z \in \Ci(c, r)$,
	\[
		k\abs{z - b}^2 + l\abs{z}^2 
	\]
	is maximized at $z = c - r$ or $z = c + r$.
	\end{fact}



			
			


	\begin{proof}[Proof to Fact~\ref{fact:distance-to-circ}]
		Observe that 
		\[
			k\abs{z - b}^2 + l\abs{z}^2 
			= (k + l)\left| z - \frac{kb}{k + l} \right|^2 + \frac{k l b^2}{k + l}.
		\]
		and distance from $\frac{kb}{k + l}$ to $z \in \Ci(c, r)$ is maximized at $z = c - r$ if $\frac{kb}{k + l} > c$ and $z = c + r$ otherwise.
	\end{proof}

    We now prove Theorem~\ref{thm:main-result-1}, \ref{thm:main-result-2}.
			
	\begin{proof}[Proof to Theorem~\ref{thm:main-result-1}]
		We first prove the first statement and show the other by the same reasoning. Invoking  Fact~\ref{fact:tight-coeff-dys} and Fact~\ref{fact:max-mod}, it suffices to show that
		\[
			\abs{\zeta_{\text{DYS}}(z_f, z_g, z_h; \alpha, \lambda)}^2 \le \rho_f^2
		\]
		for
		\begin{align*}
			z_f &\in \partial\srg{\resa{\cD_f}} = \Ci\left( C_f, R_f \right), \\
			z_g &\in \partial\srg{\resa{\cD_g}} = \Ci\left( C_g, R_g \right), \\
			z_h &\in \partial\srg{\cD_h} = \Ci\left( \frac{L_h + \mu_h}{2}, \frac{L_h - \mu_h}{2} \right)
		\end{align*}
		when $\lambda < 1 / C_f$ holds.
		We refer the readers to Fact~\ref{fact:srg-trans} and Fact~\ref{fact:cvx-srg-first} to see why $\partial\srg{\resa{\cD_f}}$, $\partial\srg{\resa{\cD_g}}$, $\srg{\cD_h}$ are given as above.

		Denoting $r = \frac{d}{1 / C_f - \lambda}$, we have
		\begin{align}
			&\abs{\zeta_{\text{DYS}}(z_f, z_g, z_h; \alpha, \lambda)}^2 \nonumber \\
			&= \abs{1 - \lambda z_f - \lambda z_g + \lambda(2 - \alpha z_h) z_f z_g}^2 \nonumber \\
			&= \abs{(1 - \lambda z_f)(1 - \lambda z_g) + \lambda(2 - \lambda - \alpha z_h) z_f z_g}^2 \nonumber \\
			&\le \left( \abs{(1 - \lambda z_f)(1 - \lambda z_g)} + \abs{\lambda(2 - \lambda - \alpha z_h) z_f z_g} \right)^2 \nonumber \\
			&\stackrel{(i)}{\le} \left( \abs{(1 - \lambda z_f)(1 - \lambda z_g)} + \lambda d \abs{z_f z_g} \right)^2 \nonumber \\
			&\stackrel{(ii)}{\le} \left( \abs{1 - \lambda z_f}^2 + \lambda d r^{-1} \abs{z_f}^2 \right)
			\left( \abs{1 - \lambda z_g}^2 + \lambda d r \abs{z_g}^2 \right),
			\label{eq:r-split-CS}
		\end{align} 
		where $(i)$ follows from $\abs{2 - \lambda - \alpha z_h} \le \max\{ \abs{2 - \lambda - \alpha\mu_h}, \abs{2 - \lambda - \alpha L_h}\} = d$ and $(ii)$ follows from Fact~\ref{fact:CS}.

		Recall that $\partial\srg{\resa{\cD_f}} = \Ci(C_f, R_f)$.
		This renders
		\begin{equation}\label{eq:main-thm-former}
			\abs{1 - \lambda z_f}^2 + \lambda d r^{-1} \abs{z_f}^2
			= \frac{\lambda}{C_f} \abs{z_f - C_f}^2 + 1 - \lambda C_f
			= 1 - \lambda \frac{C_f^2 - R_f^2}{C_f}.
		\end{equation}
		For the other term, $z_g = \frac{1}{1 + \alpha \mu_g}$ or $z_g = \frac{1}{1 + \alpha L_g}$  give the maximum, invoking Fact~\ref{fact:distance-to-circ}. Therefore,
        \begin{align}
            &\abs{1 - \lambda z_g}^2 + \lambda d r \abs{z_g}^2 \nonumber \\
            &\le
            \max \bigg\{
                \left(1 - \frac{\lambda}{1 + \alpha \mu_g}\right)^2 + \frac{\lambda d^2}{1 / C_f - \lambda} \left( \frac{1}{1 + \alpha \mu_g} \right)^2, \nonumber \\
                &\qquad\qquad\left( 1 - \frac{\lambda}{1 + \alpha L_g} \right)^2 + 
                \frac{\lambda d^2}{1 / C_f - \lambda} \left( \frac{1}{1 + \alpha L_g} \right)^2 
            \bigg\}. \label{eq:main-thm-latter}
        \end{align}
		Plugging \eqref{eq:main-thm-former} and \eqref{eq:main-thm-latter} into \eqref{eq:r-split-CS}, we obtain
		\[
			\abs{\zeta_{\text{DYS}}(z_f, z_g, z_h; \alpha, \lambda)}^2 \le \rho_f^2
		\]
		which concludes the proof for the first statement.
		The same reasoning can be applied to prove the second statement.
	\end{proof}

	\begin{proof}[Proof to Theorem~\ref{thm:main-result-2}]
		By the same reasoning in the proof of Theorem~\ref{thm:main-result-1}, it suffices to show that
		\[
			\abs{\zeta_{\text{DYS}}(z_f, z_g, z_h; \alpha, \lambda)} \le \rho
		\]
		for
		\begin{align*}
			z_f \in \partial\srg{\resa{\cD_f}}, \quad
			z_g \in \partial\srg{\resa{\cD_g}}, \quad
			z_h \in \partial\srg{\cD_h}.
		\end{align*}

		Recalling that 
		$\partial\srg{\cD_h} = \Ci\left( \frac{L_h + \mu_h}{2}, \frac{L_h - \mu_h}{2} \right)$ and $\theta = \frac{2}{4 - \alpha (\mu_h + L_h)}$, we have
		\begin{equation}\label{eq:main-thm-2-eq1}
			\abs{2 - \theta^{-1} - \alpha z_h} = \alpha \babs{z_h - \frac{L_h + \mu_h}{2}} 
			= \alpha \frac{L_h - \mu_h}{2}.
		\end{equation}
		Now, observe
		\begin{align}
			&\abs{\zeta_{\text{DYS}}(z_f, z_g, z_h; \alpha, \lambda) - (1 - \lambda \theta)}^2 \nonumber\\
			&= \lambda^2 \abs{\theta - z_f - z_g + (2 - \alpha z_h) z_f z_g}^2 \nonumber\\
			&= \lambda^2 \abs{\theta^{-1}(z_f - \theta)(z_g - \theta) 
			+ (2 - \theta^{-1} - \alpha z_h) z_f z_g}^2 \nonumber\\
			&\le \lambda^2 \left( \theta^{-1} \abs{z_f - \theta} \abs{z_g - \theta}
			+ \abs{2 - \theta^{-1} - \alpha z_h} \abs{z_f} \abs{z_g} \right)^2 \nonumber\\
			&\stackrel{(i)}{=} \lambda^2 \left( \theta^{-1} \abs{z_f - \theta} \abs{z_g - \theta}
			+ \alpha\frac{L_h - \mu_h}{2} \abs{z_f} \abs{z_g} \right)^2 \nonumber\\
			&\stackrel{(ii)}{\le} \lambda^2 \left( \theta^{-1} \abs{z_f - \theta}^2 
			+ \alpha\frac{L_h - \mu_h}{2} \abs{z_f}^2 \right) \left( \theta^{-1} \abs{z_g - \theta}^2 
			+ \alpha\frac{L_h - \mu_h}{2} \abs{z_g}^2 \right). \label{eq:main-thm-2-eq2}
		\end{align}
		Here, $(i)$ follows from \eqref{eq:main-thm-2-eq1} and $(ii)$ follows from Fact~\ref{fact:CS}.

		Invoking Fact~\ref{fact:distance-to-circ},
		\[
			\theta^{-1} \abs{z_f - \theta}^2 + \alpha\frac{L_h - \mu_h}{2} \abs{z_f}^2
		\] 
		is maximized at either $z_f = \frac{1}{1 + \alpha L_f}$ or $z_f = \frac{1}{1 + \alpha \mu_f}$.
		The first term evaluates to
		\begin{align*}
			\theta^{-1} \abs{z_f - \theta}^2 
			+ \alpha\frac{L_h - \mu_h}{2} \abs{z_f}^2 
			= \theta - \alpha \frac{2 L_f + \mu_h}{(1 + \alpha L_f)^2}
		\end{align*}
		when $z_f = \frac{1}{1 + \alpha L_f}$, and
		\begin{align*}
			\theta^{-1} \abs{z_f - \theta}^2 
			+ \alpha\frac{L_h - \mu_h}{2} \abs{z_f}^2 
			= \theta - \alpha \frac{2 \mu_f + \mu_h}{(1 + \alpha \mu_f)^2}
		\end{align*}
		when $z_f = \frac{1}{1 + \alpha \mu_f}$. Hence,
		\begin{align}
			&\theta^{-1} \abs{z_f - \theta}^2 
			+ \alpha\frac{L_h - \mu_h}{2} \abs{z_f}^2 \nonumber \\
			&\le \theta - \alpha \min \left\{ 
				\frac{2 L_f + \mu_h}{(1 + \alpha L_f)^2}, 
				\frac{2 \mu_f + \mu_h}{(1 + \alpha \mu_f)^2} \right\} \nonumber \\
			&= \theta - \alpha \nu_f. \label{eq:main-thm-2-eq3}
		\end{align}
		Similarly, we have
		\begin{equation}\label{eq:main-thm-2-eq4}
			\theta^{-1} \abs{z_g - \theta}^2 
			+ \alpha\frac{L_h - \mu_h}{2} \abs{z_g}^2 
			\le \theta - \alpha \nu_g. 
		\end{equation}
		Plugging \eqref{eq:main-thm-2-eq3} and \eqref{eq:main-thm-2-eq4} into \eqref{eq:main-thm-2-eq2} and applying the triangle inequality results in the desired bound
		\[
			\abs{\zeta_{\text{DYS}}(z_f, z_g, z_h; \alpha, \lambda)}
			\le 1 - \lambda \theta + \lambda \sqrt {\left( \theta - \alpha \nu_f \right) \left( \theta - \alpha \nu_g \right)}.
		\]
	\end{proof}

	\subsection{Comparison with previous results}
    \label{ss:comparison}
	We now compare our linear convergence rates with existing results and show that our results are not worse than the prior rates in all cases and are strictly better for most cases.

    \paragraph{Comparison with Condat and Richt{\'a}rik \cite{CondatRichtarik2022_randprox}}
    Consider problem \eqref{eq:primal} where $f \in \cF_{0, L_f}$, $g \in \cF_{\mu_g, \infty}$, and $h \in \cF_{\mu_h, L_h}$ with the constants satisfying $\mu_g > 0$ or $\mu_h > 0$, $L_f$, $L_h \in (0, \infty)$, and $\alpha \in (0, 2/L_h)$. 
    Theorem~9 of \cite{CondatRichtarik2022_randprox} with $\omega = 0$ in its formulation gives a linear convergence rate of the DYS iteration (without averaging) as follows:
    \begin{equation}\label{eq:condat}
        \rho_\mathrm{prev}^2 = \max \left\{ 
            \frac{(1 - \alpha \mu_h)^2}{1 + \alpha \mu_g},
            \frac{(1 - \alpha L_h)^2}{1 + \alpha \mu_g},
            \frac{\alpha L_f}{\alpha L_f + 2}
        \right\}.
    \end{equation}
    In the same setting, we get the following convergence rate as a direct consequence of the second part of Theorem~\ref{thm:main-result-1}:
    \begin{equation}\label{eq:our-rate}
        \rho_\mathrm{ours}^2 = 
        \max \left\{ \frac{d^2}{1 + 2\alpha\mu_g}, 
        \frac{1}{(1 + \alpha L_f)^2} \left( \alpha^2 L_f^2 + \frac{d^2}{1 + 2\alpha \mu_g} \right) \right\},
    \end{equation}
    where $d = \max\{\abs{1 - \alpha \mu_h}, \abs{1 - \alpha L_h}\}$.
    Notably, our newly derived rate always satisfies $\rho_\mathrm{ours} \leq \rho_\mathrm{prev}$, and the strict inequality $\rho_\mathrm{ours} < \rho_\mathrm{prev}$ holds whenever $\mu_g > 0$.
    For brevity, we omit the detailed calculations verifying this result.

\begin{figure}
    \centering
    \includegraphics[width=0.66\textwidth]{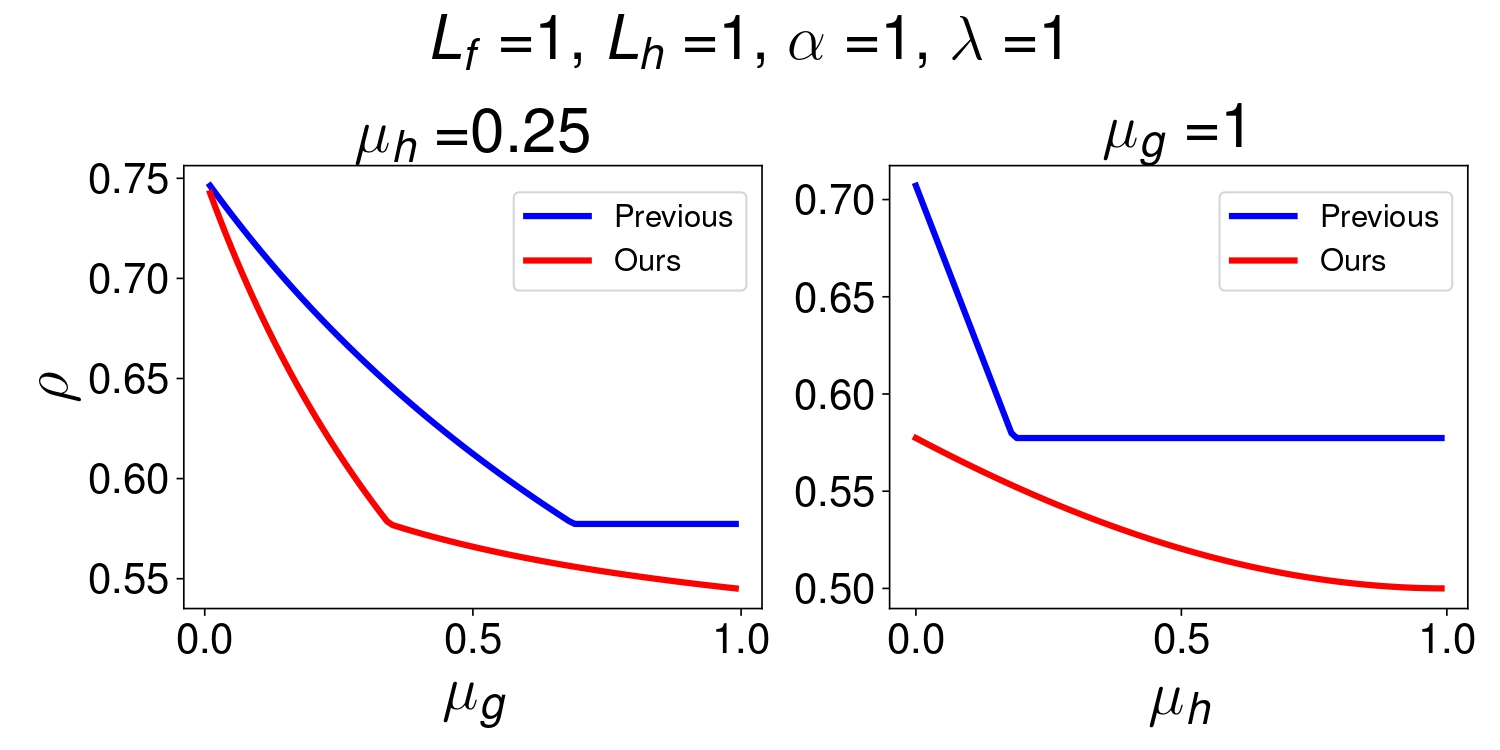}
    \caption{Convergence rate comparison between our rates and those of Condat and Richt{\'a}rik  \cite{CondatRichtarik2022_randprox}.
    With \(L_f=1\), \(L_h=1\), and \(\alpha=1\), we evaluated each contraction factor over the range of \(\mu_g\), and \(\mu_h\).}
    \label{fig:rate-comparison-1}
\end{figure}

	Figure~\ref{fig:rate-comparison-1} compares our convergence rates obtained against the prior results by Condat and Richt{\'a}rik \cite{CondatRichtarik2022_randprox} given by \eqref{eq:our-rate} and \eqref{eq:condat}, respectively. Evaluating on two different sweeps of strong convexity parameters $\mu_g$ and $\mu_h$ respectively for $f$ and $g$ in \eqref{eq:primal}, the figure shows that our analytical rates are consistently better than or match the previous rates.

    \paragraph{Comparison with Lee, Yi, and Ryu \cite{LeeYiRyu2022_convergencea}}
    We now compare our newly derived convergence rates across different settings with those implied by Theorems~3.1, 3.2, and 3.3 of \cite{LeeYiRyu2022_convergencea}. 
    As in the previous section, we omit the detailed computations supporting the comparisons.

    In the case where $f \in \cF_{\mu_f, L_f}$, $g \in \cF_{0, \infty}$, and $h \in \cF_{0, L_h}$, with $\alpha L_h < 4$ and $\lambda < 2 - \frac{\alpha L_h}{2}$,
    Theorem~3.1 of \cite{LeeYiRyu2022_convergencea} implies iterations with $\opT_{\partial f, \partial g, \nabla h, \alpha, \lambda}$ converges linearly with a rate 
    \[
        \rho_\mathrm{prev} = 1 - \frac{2\lambda}{4 - \alpha L_h} + \lambda \sqrt{ \frac{2}{4 - \alpha L_h}  \left( \frac{2}{4 - \alpha L_h} - \frac{2 \alpha \mu_f}{\alpha^2 L_f^2 + 2 \alpha \mu_f + 1} \right)}.
    \]
    Meanwhile, Theorem~\ref{thm:main-result-2} gives a linear convergence rate of 
    \[
        \rho_\mathrm{ours} = 1 - \frac{2\lambda}{4 - \alpha L_h} + \lambda \sqrt{ \frac{2}{4 - \alpha L_h}  \left( \frac{2}{4 - \alpha L_h} - 
        \alpha \min\left\{ \frac{2\mu_f}{(1 + \alpha\mu_f)^2}, \frac{2L_f}{(1 + \alpha L_f)^2} \right\}
        \right)}.
    \]
    It holds that $\rho_\mathrm{ours} \le \rho_\mathrm{prev}$, with the strict inequality $\rho_\mathrm{ours} < \rho_\mathrm{prev}$ as long as $\mu_f > 0$.

    Now, consider the setting where $f \in \cF_{0, L_f}$, $g \in \cF_{\mu_g, \infty}$, and $h \in \cF_{0, L_h}$. 
    Theorem~3.2 of \cite{LeeYiRyu2022_convergencea} implies iterations using $\opT_{\partial f, \partial g, \nabla h, \alpha, \lambda}$ renders a linear convergence with rate
    \[
        \rho_\mathrm{prev} = 
        \sqrt{1 - 2 \lambda \alpha \min\left\{ 
             \frac{(2-\lambda) \mu_g}{(1 + \alpha^2 L_f^2)(2-\lambda+2\alpha \mu_g)},
            \frac{( 2 - \lambda)(\mu_g + L_f) + 2\alpha \mu_g L_f }
            {(1 + \alpha L_f)^2(2 - \lambda + 2 \alpha \mu_g)}
        \right\}}.
    \]
    On the other hand, the second part of Theorem~\ref{thm:main-result-1} implies a linear convergence rate of 
    \[
        \rho_\mathrm{ours} = 
        \sqrt{1 - 2 \lambda \alpha \min\left\{ 
            \frac{(2-\lambda)\mu_g}{2-\lambda+2\alpha \mu_g},
            \frac{( 2 - \lambda)(\mu_g + L_f) + 2\alpha \mu_g L_f }
            {(1 + \alpha L_f)^2(2 - \lambda + 2 \alpha \mu_g)}
        \right\}}.
    \]
    As before, $\rho_\mathrm{ours} \le \rho_\mathrm{prev}$ holds, and we have the strict inequality $\rho_\mathrm{ours} < \rho_\mathrm{prev}$ if $(2 - \lambda)(1 - 2\alpha\mu_g + \alpha^2 L_f^2) + 2\alpha\mu_g(1 + \alpha^2 L_f^2) > 0$ and $\mu_g > 0$.
		
    For the last case, consider $f \in \cF_{0, L_f}$, $g \in \cF_{0, \infty}$, and $h \in \cF_{\mu_h, L_h}$.
    Theorem~3.3 of \cite{LeeYiRyu2022_convergencea} implies $\opT_{\partial f, \partial g, \nabla h, \alpha, \lambda}$ renders the fixed point iteration with a linear convergence rate of
    \[
        \rho_\mathrm{prev} = 
        \sqrt{1 - 2\lambda\alpha\min\left\{ 
            \frac{ \mu_h \left( 1 - \frac{\alpha L_h}{2(2-\lambda)} \right)}{1 + \alpha^2 L_f^2},
            \frac{L_f + \mu_h \left( 1 - \frac{\alpha L_h}{2(2-\lambda)} \right)}{(1 + \alpha L_f)^2}
            \right\}}.
    \]
    In contrast, the second part of Theorem~\ref{thm:main-result-1} implies the same iterations linearly converge with a rate of
    \[
        \rho_\mathrm{ours} = 
        \sqrt{1 - 2\lambda\alpha \min\left\{ 
            \xi, 
            \frac{L_f + \xi}{(1 + \alpha L_f)^2}
        \right\}}
    \]
    denoting
    \[
        \xi = \min\left\{ \mu_h\left( 1 - \frac{\alpha \mu_h}{2(2 - \lambda)} \right), L_h\left( 1 - \frac{\alpha L_h}{2(2 - \lambda)} \right) \right\}.
    \]
    Again, it holds that $\rho_\mathrm{ours} \le \rho_\mathrm{prev}$, with the strict inequality $\rho_\mathrm{ours} < \rho_\mathrm{prev}$ whenever $\mu_h > 0$.

	\begin{figure}
		\centering
		\includegraphics[width=\textwidth]{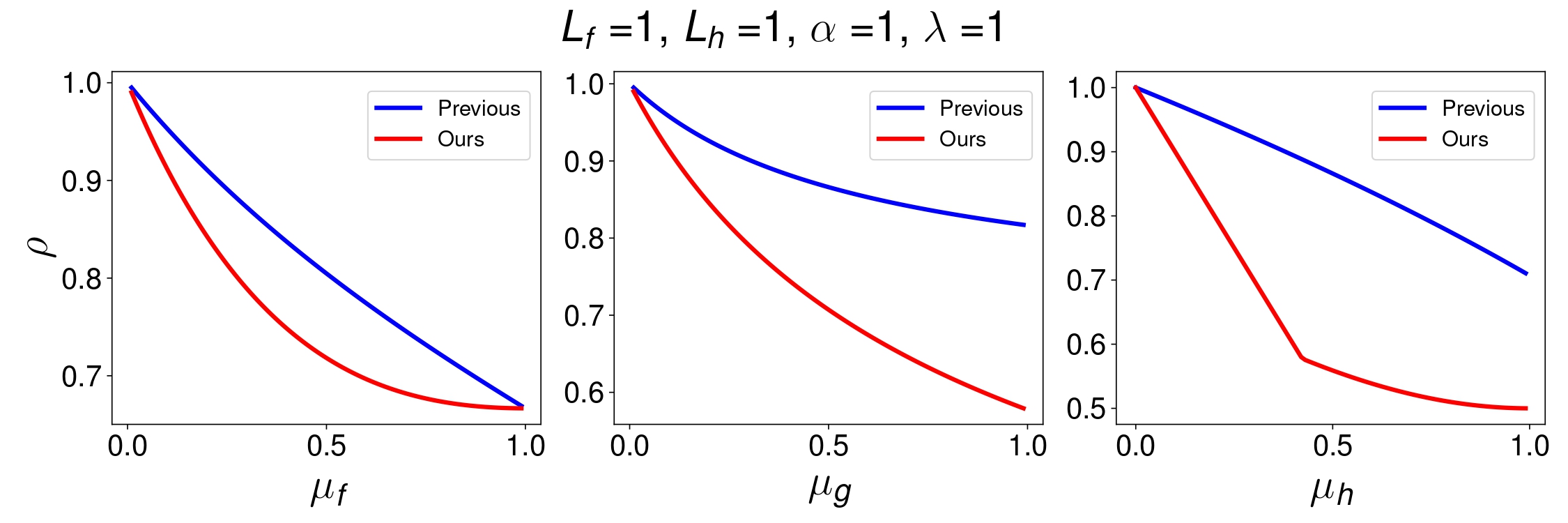}
		\caption{Convergence rate comparison between our rates and those of Lee, Yi, and Ryu \cite{LeeYiRyu2022_convergencea}. With \(L_f = 1\), \(L_h = 1\), \(\alpha=1\), and \(\lambda=1\), we evaluated each contraction factor over the range of \(\mu_f\), \(\mu_g\), and \(\mu_h\), each corresponding to the first, second, and third comparison cases respectively.}
		\label{fig:rate-comparison-2}
	\end{figure}

Figure~\ref{fig:rate-comparison-2} compares our convergence rates against the prior analytical rates from Lee, Yi, and Ryu \cite{LeeYiRyu2022_convergencea}. 
Again, our analytical rates provided by Theorems~\ref{thm:main-result-1} and \ref{thm:main-result-2} consistently outperform the prior rates across different settings.

	\paragraph{Comparison with FBS}
	\begin{figure}
		\centering
		\includegraphics[width=0.67\textwidth]{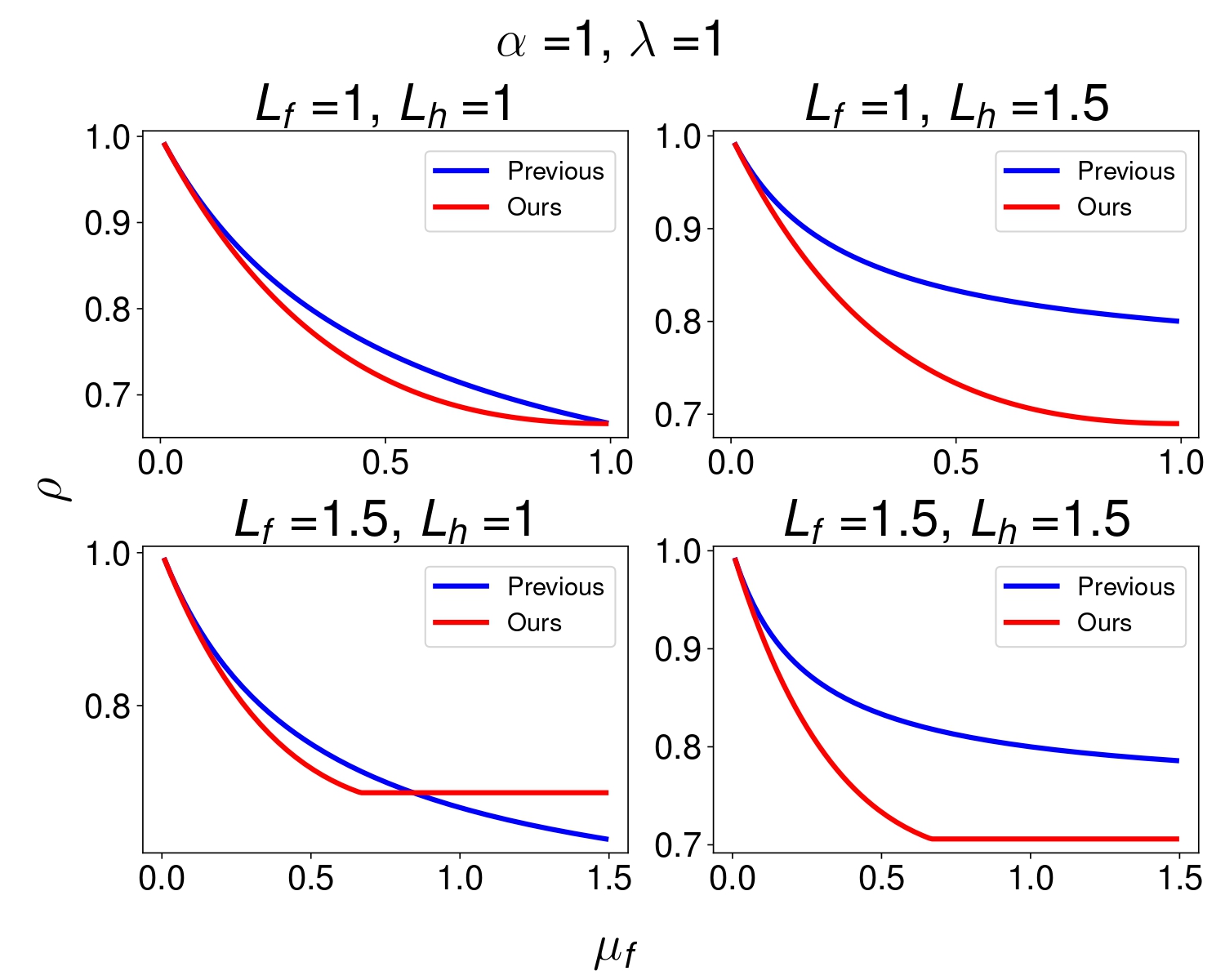}
		\caption{Comparison of our linear convergence rates and contraction factors of FBS provided in \cite{guo2021onthelinear}. 
		We consider FBS applied to \(f + g\) and \(h\) in \eqref{eq:primal}.
		The Lipschitz coefficients are computed across four different choices of \(L_f\) and \(L_h\), and for each setting, we evaluate the contraction factor by varying \(\mu_f\) across a range of values \((0, L_f)\).
		}
		\label{fig:rate-comparison-fbs}
	\end{figure}
	The linear convergence rates in this work can be compared to  known contraction factors of FBS by viewing the sum of two objective functions as a single function.
	In particular, we consider the setup where \(f \in \cF_{\mu_f, L_f}\), \(g \in \cF_{0, \infty}\), and \(h \in \cF_{0, L_h}\), and apply two-operators splitting with respect to the combined objective \(f + g\) and function \(h\) in \eqref{eq:primal}.
	Under these conditions, we have \(f + g \in \cF_{\mu_f, \infty}\), which allows to use the contraction factor of FBS provided in \cite{guo2021onthelinear}.
	Figure~\ref{fig:rate-comparison-fbs} shows that our contraction factors are generally better than the ones provided in \cite{guo2021onthelinear}.

\section{Discussion and conclusion}
 The reduction of Fact~\ref{fact:tight-coeff-dys} allows us to obtain the Lipschitz coefficients Theorems~\ref{thm:main-result-1} and \ref{thm:main-result-2} by characterizing the maximum modulus of
	\begin{align*}
	\cZ^{\text{DYS}}_{\cA, \cB, \cC, \alpha, \lambda} = 	\left\{ \zeta_{\text{DYS}}(z_f, z_g, z_h; \alpha, \lambda) \,\middle|\, 
		z_f \in \srg{\resa{\cD_f}},
		z_g \in \srg{\resa{\cD_g}},
		z_h \in \srg{\cD_h}
		\right\},
	\end{align*}
 where $\zeta_{\text{DYS}} = 1 - \lambda z_B + \lambda z_A (2z_B - 1 - \alpha z_C z_B) $ is a relatively simple polynomial of three complex variables.
This only requires elementary mathematics, and it is considerably easier than directly analyzing
	\begin{equation*}\label{eq:very-complex-set}
        \left\{ 
        \frac{\norm{\opT x - \opT y}}{\norm{x - y}} \,\middle|\,
        \opT \in \opT_{\cD_f, \cD_g, \cD_h, \alpha, \lambda}, \,
        x, y \in \dom\opT, \, x \ne y
        \right\}.
	\end{equation*}
Furthermore, by obtaining tighter bounds on the set $\cZ^{\text{DYS}}_{\cA, \cB, \cC, \alpha, \lambda}$, one can improve upon the contraction factors presented in this work.




The explicit and simple description of $\cZ^{\text{DYS}}_{\cA, \cB, \cC, \alpha, \lambda}$ allows one to investigate it in a numerical and computer-assisted manner. Sampling points from $\cZ^{\text{DYS}}_{\cD_f, \cD_g, \cD_h, \alpha, \lambda}$ is straightforward, and doing so provides a numerical estimate of the maximum modulus. For example, Figure~\ref{fig:dys_subdiff_example} depicts $\cZ^{\text{DYS}}_{\cD_f, \cD_g, \cD_h, \alpha, \lambda}$ with a specific choice of $\mu_f$, $\mu_g$, $\mu_h$, $L_f$, $L_g$, $L_h$, $\alpha$, and $\lambda$. It shows that $\rho_g$, the contraction factor of Theorem~\ref{thm:main-result-1}, is valid but not tight; the gap between $\cZ^{\text{DYS}}_{\cD_f, \cD_g, \cD_h, \alpha, \lambda}$ and $\Ci(0, \rho_g)$ indicates the contraction factor has room for improvement.
Interestingly, if we modify the proof of Theorem~\ref{thm:main-result-1} to choose $r$ in \eqref{eq:r-split-CS} more carefully, we seem to obtain a tight contraction factor in the instance of Figure~\ref{fig:dys_subdiff_example}.
    Specifically, when we numerically minimize $\rho$ as a function of $r$, we observe that $\Ci(0, \rho)$ touches $\cZ^{\text{DYS}}_{\cD_f, \cD_g, \cD_h, \alpha, \lambda}$ in Figure~\ref{fig:dys_subdiff_example} and the contact indicates tightness.

	\begin{figure}[h]
		\centering
		\includegraphics[width=0.7\textwidth]{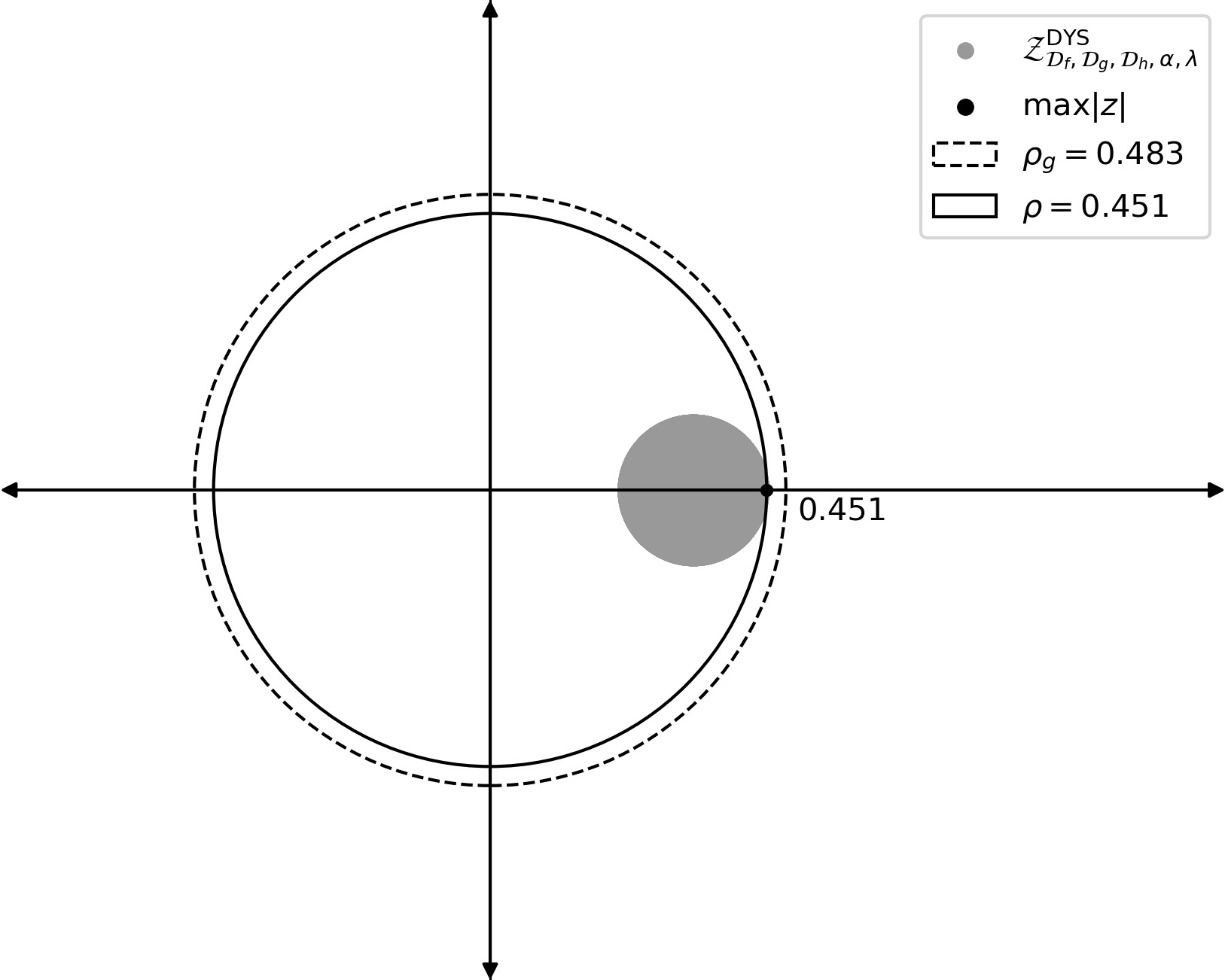}
		\caption{
		$\cZ^{\text{DYS}}_{\cD_f, \cD_g, \cD_h, \alpha, \lambda}$ with $\Ci(0, \rho_f)$, $\Ci(0, \rho_g)$, and $\Ci(0, \rho)$,
		where $\mu_f = 0.7$, $\mu_g = 2$, $\mu_h = 0.8$, $L_f = 1.5$, $L_g = 3$, $L_h = 1.3$, $\alpha = 0.9$, and $\lambda = 1$.}
		\label{fig:dys_subdiff_example}
	\end{figure}

\section*{Acknowledgments}
This work was supported by the Samsung Science and Technology Foundation (Project Number SSTF-BA2101-02). 

\bibliographystyle{elsarticle-num} 
\bibliography{main.bib}

\end{document}